\documentclass[secnum,leqno]{rims-bessatsu}
\usepackage{amssymb, amsmath}
\usepackage{graphics}
\usepackage[dvips]{graphicx}
\usepackage{eepic}
\usepackage{epic}
\newtheorem{thm}{Theorem}[section]

\newtheorem{lmm}[thm]{Lemma}

\newcommand{\cd}{\, \cdot\, }

\newcommand{\Rn}{{\mathbb R}^n}

\newcommand{\la}{\lambda}

\newcommand{\e}{\varepsilon}

\renewcommand{\Pi}{\varPi}

\renewcommand{\epsilon}{\varepsilon}

\newcommand{\R}{{\mathbb R}}

\newcommand{\ltube}{{\mathcal T}_{\small\la^{-1/2}\normalsize}(\gamma)}
\newcommand{\dtube}{{\mathcal T}_{\delta}(\gamma)}

\newcommand{\rball}{B_r(x)}
\newcommand{\lball}{B_{\la^{-1}}(x)}

\newcommand{\vertiii}[1]{{\left\vert\kern-0.25ex\left\vert\kern-0.25ex\left\vert #1 
    \right\vert\kern-0.25ex\right\vert\kern-0.25ex\right\vert}}

\newcommand{\1}{{\rm 1\hspace*{-0.4ex}%
\rule{0.1ex}{1.52ex}\hspace*{0.2ex}}}

\theoremstyle{definition}

\theoremstyle{remark}
\title{Localized $L^p$-estimates for eigenfunctions: II}
\author{Christopher D.  \textsc{Sogge}\footnote{Johns Hopkins University, Baltimore, MD
21218, USA.\newline e-mail: \texttt{sogge@jhu.edu}}}
\AuthorHead{Christopher D. Sogge}         
\classification{58J51 (35P15 35R01 42B37)}
\support{The author was supported in part by the NSF grant DMS-1361476.}  
\keywords{\textit{Eigenfunctions, localization, quantum ergodicity}}         
\VolumeNo{x}           
\YearNo{201x}           
\PagesNo{000--000}      
\communication{Received April 20, 201x.
Revised September 11, 201x.}      
\begin{document}
%

\maketitle

\begin{abstract}      
If $(M,g)$ is a compact Riemannian manifold of dimension $n\ge 2$ we give necessary and sufficient conditions for improved $L^p(M)$-norms of eigenfunctions for
all $2<p\ne p_c=\tfrac{2(n+1)}{n-1}$, the critical exponent.  Since improved $L^{p_c}(M)$ bounds imply improvement
all  other exponents, these conditions are necessary for improved bounds for the critical space.  We also show
that improved $L^{p_c}(M)$ bounds are valid if these conditions are met and if the half-wave operators, $U(t)$, 
have no caustics
when $t\ne 0$.  The problem of finding a necessary and sufficient condition for $L^{p_c}(M)$ improvement remains 
an interesting open problem.
\end{abstract}

\section{Local and Global Estimates of Eigenfunctions}

Let $(M,g)$ be a compact Riemannian manifold of dimension $n\ge2$.  If $\Delta_g$ is the associated
Laplace-Beltrami operator, we shall consider the $L^2$-normalized eigenfunctions satisfying
\begin{equation}\label{1.1}
-\Delta_g e_\la(x)=\la^2 e_\la(x) \quad \text{and } \, \, \int_M|e_\la|^2 \, dV_g=1,
\end{equation}
with $dV_g$ denoting the Riemannian volume element.  The purpose of this paper is to show
that one has favorable $L^p(M)$ bounds for the $e_\la$ if and only if there is not saturation
of $L^p$ or $L^2$ norms taken over very small sets that shrink as $\la\to \infty$.  The sets depend
on whether $p>2$ is larger or smaller than the critical exponent $p=p_c=\tfrac{2(n+1)}{n-1}$.
For $p>p_c$ there are improved $L^p(M)$ bounds if and only if $L^p$ or $L^2$ norms over geodesic balls
of radius $\la^{-1}$ are not saturated, while for $2<p<p_c$ one obtains improvement if and only if
there is not saturation of these norms taken over $\la^{-1/2}$ tubular neighborhoods of unit length geodesics.
We shall also show that we have improved $L^{p_c}(M)$ bounds if we have improved $L^p(M)$ estimates
for all $p\in (2,\infty]\, \backslash \, \{p_c\}$ and if the half-wave operators $e^{-it\sqrt{-\Delta_g}}$
have no caustics for non-zero times (see \S 2).

Recall that in \cite{Sef} we showed that for $p>2$ one always has the universal bounds
\begin{equation}\label{1.2}
\|e_\la\|_{L^p(M)}=O(\la^{\mu(p)}),
\end{equation}
if
\begin{equation}\label{1.3}
\mu(p)=
\begin{cases}
\tfrac{n-1}2(\tfrac12-\tfrac1p), \quad \, \, \, \,  2<p\le p_c,
\\
n(\tfrac12-\tfrac1p)-\tfrac12, \quad p_c\le p\le \infty.
\end{cases}
\end{equation}
These norms are saturated on the round sphere $S^n$; however, for generic manifolds $(M,g)$
it was shown by Zelditch and the author \cite{SZDuke} that
$\|e_\la\|_{L^p(M)}=o(\la^{\mu(p)})$ if $p>p_c$.  Whether one generically has
improvements for $2<p< p_c$
or better yet for $p=p_c$
remains an open problem.

Let us now state two of our main results.  If $B_r(x)$ denotes a geodesic ball of radius
$0<r<\text{Inj }M$ (the injectivity radius of $(M,g)$), then the first is the following.

\begin{thm}\label{thm1.1}  The following are equivalent:
\begin{align}
\|e_\la\|_{L^p(M)}&=o(\la^{\mu(p)}), \quad \,   \text{for all } \, \, p>p_c,
\label{1.4}
\\
\sup_{x\in M} \|e_\la\|_{L^p(B_{\la^{-1}}(x))}&=o(\la^{\mu(p)}), \quad \text{for some } \, \, p> p_c,
\label{1.5}
\\
\sup_{x\in M} \|e_\la\|_{L^2(B_{\la^{-1}}(x))}&=o(\la^{-\frac12}).
\label{1.6}
\end{align}
\end{thm}

It was shown in \cite{SHR} and \cite{Scrit} that on any $(M,g)$ one has $\|e_\la\|_{L^2(B_r(x))}\le Cr^{\frac12}$
with $C=C_M$ for $\la^{-1}\le r\le \text{Inj }M$, and so \eqref{1.6} just involves improving this
universal estimate in the extreme case where the radius $r=\la^{-1}$ is the frequency of the eigenfunction.

If $\varPi$ denotes the space of all unit length geodesics on $M$ and if $\dtube$ denotes
the $\delta$ tubular neighborhood about a given $\gamma\in \varPi$, i.e.,
$$\dtube=\{x\in M: \, d_g(x,\gamma)<\delta\},$$
with $d_g(\cd, \cd)$ denoting the Riemannian distance function, then we also have the following
complementary result.

\begin{thm}\label{thm1.2}
The following are equivalent:
\begin{align}
\|e_\la\|_{L^p(M)}&=o(\la^{\mu(p)}), \quad \, \text{for all} \, \, 2<p<p_c,
\label{1.7}
\\
\sup_{\gamma\in \varPi}\|e_\la\|_{L^p(\ltube)}&=o(\la^{\mu(p)}), \quad \text{for some } \, \, 2<p<p_c,
\label{1.8}
\\
\sup_{\gamma \in \varPi}\|e_\la\|_{L^2(\ltube)}&=o(1).
\label{1.9}
\end{align}
\end{thm}

The scales in the two theorems are very natural.  As far as the first one goes, recall that
the $L^2$-normalized zonal functions, $Z_\la$, on $S^n$ saturate the $L^p$ bounds in \eqref{1.2} for
$p\ge p_c$ (see \cite{Sthesis}).  This is because, modulo lower order terms, the $Z_\la$
behave like an oscillatory factor times $r^{-\frac{n-1}2}$ if $r$ is larger than $\la^{-1}$,
with $r$ denoting the minimum of the distance to the two poles on $S^n$, and $|Z_\la|\approx \la^{\frac{n-1}2}$
if $r$ is smaller than a fixed multiple of $\la^{-1}$.  Using this fact, a simple calculation
shows that the quantities in the left side of 
\eqref{1.4}--\eqref{1.5} and \eqref{1.6} are $\Omega(\la^{\mu(p)})$ and $\Omega(\la^{-\frac12})$, respectively,
assuming in the former case, that $p\ge p_c$.
Similarly, if we write $S^n$ as $\{x\in {\mathbb R}^{n+1}: \, x_1^2+\cdots + x_{n+1}^2=1\}$ and if
the highest weight spherical harmonics, $Q_\la$, are given by $\la^{\frac{n-1}4}(x_1+ix_2)^\la$,
where $\la=\la_k=\sqrt{k(k+n-1)}$, $k\in {\mathbb N}$, then these eigenfunctions have $L^2(S^n)$ norms which are comparable
to one, and, moreover, the quantities in \eqref{1.7}--\eqref{1.8} and \eqref{1.9}
are $\Omega(\la^{\mu(p)})$ and $\Omega(1)$, respectively, provided that $2<p\le p_c$.
Thus, the $Z_\la$ have the largest possible $L^2$ or $L^p$, $p\ge p_c$, mass in balls of radius
$\la^{-1}$ about either of the two poles, $\pm(0,\dots,0,1)$, on $S^n$, while the $Q_\la$
have the largest possible $L^2$ or $L^p$, $2<p\le p_c$, mass in tubes of radius $\la^{-1/2}$ of the equator where $0=x_3=\cdots=x_{n+1}$.

To verify Theorem~\ref{thm1.1}, first notice that \eqref{1.4} trivially implies \eqref{1.5}.  Also,
since $\mu(p)=n(\tfrac12-\tfrac1p)-\tfrac12$ for $p\ge p_c$, by H\"older's inequality, \eqref{1.5}
implies \eqref{1.6}.  To prove the nontrivial part of the theorem, which says that
\eqref{1.6} implies \eqref{1.4}, we recall the following recent result of the author from \cite{SHR},
which says that there is a uniform constant $C=C(M,g)$ so that for all $p>2$
\begin{equation}\label{1.10}
\|e_\la\|_{L^p(\rball)}\le Cr^{-\frac12}\la^{\mu(p)}\|e_\la\|_{L^2(B_{2r}(x))}, 
\quad \la \ge 1, \, \, \, \la^{-1}\le r<\text{Inj }M.\end{equation}
See Hezari and Rivi\`ere~\cite{HezR} for earlier related work.

We shall  use the fact that the special case of \eqref{1.10} with $r=\la^{-1}$ and $p=\infty$ implies that
\begin{equation}\label{1.11}
\|e_\la\|_{L^\infty(M)}\le C\la^{\mu(\infty)} \, 
\bigl(\la^{\frac{1}2}\sup_{x\in M}\|e_\la\|_{L^2(\lball)}\bigr).
\end{equation}
Since for $p_c<p\le \infty$,
$$\|e_\la\|_{L^p(M)}\le \|e_\la\|_{L^\infty(M)}^{\theta(p)}\, \|e\|_{L^{p_c}(M)}^{1-\theta(p)}, 
\quad \text{if } \, \, \theta(p)=\tfrac{p-p_c}p,
$$
we see that \eqref{1.2} and \eqref{1.11} yield
\begin{multline}\tag{1.11$'$}\label{1.11'}
\|e_\la\|_{L^p(M)}\le C \bigl(\la^{\mu(\infty)}\sup_{x\in M}\la^{\frac12}\|e_\la\|_{L^2(\lball)}\bigr)^{\theta(p)}
\times
\bigl(\la^{\mu(p_c)}\|e_\la\|_{L^2(M)}\bigr)^{1-\theta(p)}
\\
=C\la^{\mu(p)}\bigl(\la^{\frac12}\sup_{x\in M}\|e_\la\|_{L^2(\lball)}\bigr)^{\theta(p)}, \quad
\text{if } \, p>p_c,
\end{multline}
due to the fact that the eigenfunctions are $L^2$-normalized, and, by \eqref{1.3},
$$\mu(p)=\theta(p)\mu(\infty)+(1-\theta(p))\mu(p_c), \quad \text{if } \, p>p_c.$$
Clearly \eqref{1.11'} shows that \eqref{1.6} implies \eqref{1.4}, which completes the proof
of Theorem~\ref{thm1.1}.

To verify Theorem~\ref{thm1.2}, we first note that of course \eqref{1.7} implies \eqref{1.8}.  Also,
since $\mu(p)=\tfrac{n-1}2(\tfrac12-\tfrac1p)$ for $2<p\le p_c$, by H\"older's inequality, \eqref{1.9} follows from
\eqref{1.8}.  The nontrivial part, which is that \eqref{1.9} implies \eqref{1.7} was first established
in the two dimensional case by the author in \cite{SKN} following earlier related partial results 
of Bourgain~\cite{Bourgainef}.  In this work we called the quantity in the left side of \eqref{1.9}
the ``Kakeya-Nikodym'' norm of $e_\la$.  For higher dimensions $n\ge 3$, the fact that \eqref{1.9}
implies \eqref{1.7} is due to Blair and the author in \cite{BSJ}, and refinements were made
in \cite{BSAPDE}, \cite{BSK15} and \cite{BSTop}.  Zelditch and the author showed in \cite{SZStein}
that when $n=2$ if $(M,g)$ has nonpositive curvature the Kakeya-Nikodym norms are $o(1)$ and hence
one has \eqref{1.7}, and this result was extended to higher dimensions by Blair and the author in 
\cite{BSJ}.

\section{Improved Bounds for the Critical Space}
The problem of showing that in certain cases one has $\|e_\la\|_{L^{p_c}(M)}=o(\la^{\mu(p_c)})$ is 
much more subtle than showing that either \eqref{1.4} or \eqref{1.7} is valid.  Indeed,
as is well known (see e.g. \cite{SHR}) improvements for the critical space imply ones for all the
other spaces.  For $2<p<p_c$ one just uses H\"older's inequality, while for $p>p_c$ one obtains
improved $L^p$ bounds from improved ones for the critical space via Sobolev/Bernstein inequalities.
Indeed, if $\rho$ as in \eqref{2.3} equals one at the origin and has compactly supported 
Fourier transform then $\rho(\la^{-1}(\la-\sqrt{-\Delta_g}))e_\la=e_\la$, and, by
the arguments in \S 5.3, of \cite{SFIO}, this operator has a kernel which, for every $N=1,2,3,\dots$,
is $O(\la^n(1+\la d_g(x,y))^{-N})$ and so, by Young's inequality,
$\|\rho(\la^{-1}(\la+\sqrt{-\Delta_g}))\|_{L^p(M)\to L^q(M)}$ is $O(\la^{n(\frac1p-\frac1q)})$ for $p\le q$, 
Using this fact and \eqref{1.3}, one immediately sees that improved $L^{p_c}(M)$ bounds
lead to improved  $L^p(M)$ bounds for all $p>p_c$.

Thus, by Theorems \ref{thm1.1} and \ref{thm1.2},
if $\|e_\la\|_{L^{p_c}(M)}=o(\la^{\mu(p_c)})$ one must have \eqref{1.6} and \eqref{1.9}.  An interesting
question would be if these two necessary conditions for improved critical space bounds are sufficient.
Although we cannot answer this question we can adapt arguments from \cite{Scrit} to obtain
the following partial result.

\begin{thm}\label{theorem2.1}
Suppose that \eqref{1.6} and \eqref{1.9} are valid.  Suppose further that if $P=\sqrt{-\Delta_g}$
the half-wave operators,
\begin{equation}\label{2.1}U(t)=e^{-itP},\end{equation}
have no caustics when $t\ne 0$.  We then have
\begin{equation}\label{2.2}
\|e_\la\|_{L^{p_c}(M)}=o(\la^{\mu(p_c)}), \quad p_c=\tfrac{2(n+1)}{n-1}.
\end{equation}
\end{thm}

The assumption that the half-wave operators in \eqref{2.1} have no caustics for nonzero times is equivalent to the assumption
that $(M,g)$ has no conjugate points.  This is always the case if $(M,g)$ has nonpositive curvature, and so Theorem~\ref{theorem2.1}
partly generalizes the results from \cite{Scrit} where it was shown that $\|e_\la\|_{L^{p_c}(M)}=O((\log \log \la)^{-\sigma_n})$ for
some $\sigma_n>0$ if $(M,g)$ has nonpositive curvature. 

To prove \eqref{2.2} we shall need to eventually use operators that reproduce eigenfunctions.  To this end fix
\begin{equation}\label{2.3}
\rho \in {\mathcal S}({\mathbb R}) \quad
\text{with } \, \, \rho(0)=1 \quad \text{and } \, \, \Hat \rho(t)=0 \, \, \,
\text{if } \, \, |t|\ge 1.
\end{equation}
Then
\begin{equation}\label{2.4}
\rho(T(\la-P))e_\la = e_\la \quad \text{if } \, \, T\ge1,
\end{equation}
and
\begin{equation}\label{2.5}
\rho(T(\la-P))=\frac1{2\pi T}\int \Hat \rho(t/T) e^{i\la t} e^{-itP} \, dt.
\end{equation}
By the last part of \eqref{2.3} the integrand vanishes when $|t|\ge T$.

We shall require the following pointwise estimates for the kernels of these
operators which make use of our assumption that $e^{-itP}$ has no caustics
when $t\ne 0$.

\begin{lmm}\label{lemma2.2}  Fix $\rho$ as in \eqref{2.3} and assume that
$U(t)$ has no caustics for $t\ne 0$.  Then the kernel of the operator in \eqref{2.5}
satisfies
\begin{equation}\label{2.6}
\bigl|\rho\bigl(T(\la-P)\bigr)(x,y)\bigr| \le C\bigl(\la/d_g(x,y)\bigr)^{\frac{n-1}2}
+C_T\la^{\frac{n-1}2}, \quad \text{if } \, \, T\ge 1,
\end{equation}
where $C$ is a uniform constant depending only on $\rho$ and $(M,g)$, while
$C_T$ also depends on the parameter $T$.
\end{lmm}

\begin{proof}[Proof of Lemma~\ref{lemma2.2}]
To prove this we may assume for simplicity that the injectivity radius of 
$(M,g)$ is ten or more after possibly rescaling the metric which just has the
effect of changing the eigenvalues of $P$ by a fixed factor.

Fix $\beta\in C^\infty_0(\R)$ satisfying $\beta(t)=1$ if $|t|\le 1$ and
$\text{supp }\beta \subset (-2,2)$.  By \eqref{2.5} we can
then write
\begin{multline}\label{2.7}
\rho\bigl(T(\la-P)\bigr)(x,y)
=\frac1{2\pi T} \int \beta(t)\Hat \rho(t/T) e^{i\la t} \, U(t;x,y)\, dt
\\
+\frac1{2\pi T} \int \bigl(1- \beta(t)\bigr) \Hat \rho(t/T) e^{i\la t} \, U(t;x,y)\, dt
=I + II,
\end{multline}
with $U(t;x,y)$ denoting the kernel of $U(t)=e^{-itP}$.

Since we are assuming that $T\ge1$ it is well known and not difficult to prove
that the first term here satisfies the uniform bounds
\begin{equation}\label{2.8}
|I|\le C\bigl(\la/d_g(x,y)\bigr)^{\frac{n-1}2}.
\end{equation}
To prove this one uses the fact that our assumption about the injectivity radius
means that we can obtain a parametrix for $U(t)$ for $|t|\le 2$ as in 
\cite{Hspec} or in
Theorem 4.1.2
in \cite{SFIO}.  Using this fact it is not difficult to modify the proof of Lemma 5.1.3 in
\cite{SFIO} and use this parametrix along with a simple stationary phase argument
to obtain the uniform bounds \eqref{2.8}.

Due to \eqref{2.8}, the proof of \eqref{2.6} would be complete if we could show that
the last term in \eqref{2.7} satisfies
\begin{equation}\label{2.9}
|II|\le C_T\la^{\frac{n-1}2}.
\end{equation}

To prove this, let us recall some basic facts about the operators $U(t)$.  First they are
Fourier integral operators whose canonical relations are given by
\begin{equation}\label{2.10}
{\mathcal C}=\{(t,x,\tau,\xi, y,\eta): \, \tau=-p(y,\eta), \, \, (x,\xi)=\chi_t(y,\eta)\},
\end{equation}
where $\chi_t: T^*M\backslash 0\to T^*M\backslash 0$ denotes geodesic flow on the cotangent bundle.  If we fix
the time $t$ then the canonical relation of $U(t): \, {\mathcal D}'(M)\to {\mathcal D}'(M)$ is then
\begin{equation}\label{2.11}
{\mathcal C}_t=\{(x,\xi,y,\eta): \, (x,\xi)=\chi_t(y,\eta)\}.
\end{equation}
The assumption that $U(t)$, $t\ne 0$, has no caustics means that the projection from ${\mathcal C}_t$ to $M\times M$
has a differential with rank $2n-1$ everywhere.  The image of this projection then is an immersed hypersurface
of codimension one.

This all means that for $t$ near a given $t_0\ne 0$, modulo smooth errors, we can write the kernel of $U(t)$ as a finite
sum of Fourier integrals which in local coordinates are of the form
\begin{equation}\label{2.12}
\int_{\Rn} e^{i\varphi(x,y,t,\xi)} \, a(t,x,y,\xi) \, d\xi,
\end{equation}
where $a\in S^0$ is a symbol of order zero and $\varphi$ solves the eikonal equation
\begin{equation}\label{2.13}
\varphi'_t =-p(x,\nabla_x\varphi),
\end{equation}
on the support of $a$ with
\begin{equation}\label{2.14}
p(x,\xi)=\sqrt{\sum g^{jk}(x)\xi_j\xi_k}
\end{equation}
being the principal symbol of $P=\sqrt{-\Delta_g}$.  Here $(g^{jk}(x))=(g_{jk}(x))^{-1}$ is the cometric written in our
local coordinate system.  The phase $\varphi$ is real and smooth away from $\xi=0$ and it is homogeneous of degree one
in the $\xi$ variable.  Additionally, on $\text{supp }a$ we have that
$\nabla_x\varphi\ne 0$ if $\nabla_\xi \varphi=0$ and $\xi\ne 0$.  Consequently,
by \eqref{2.13}--\eqref{2.14}, we have on the support of $a$ that
\begin{equation}\label{2.15}
\bigl\langle \tfrac{\xi}{|\xi|},\nabla_\xi\bigr\rangle \varphi'_t\ne 0, \quad \text{if } \, \varphi'_\xi=0.
\end{equation}

To use this we recall that since we are assuming that the projection from ${\mathcal C}_t$ to $M\times M$ has 
a differential of rank $2n-1$ everywhere, we must have that on $\text{supp }a$
\begin{equation}\label{2.16}
\text{Rank }\frac{\partial^2\varphi}{\partial \xi_j\partial \xi_k}\equiv n-1 \quad \text{if } \, \, \nabla_\xi \varphi=0, \, \, \xi\ne 0.
\end{equation}
(See e.g., Proposition 6.1.5 in \cite{SFIO}.)
Since $\varphi$ is homogeneous of degree one in $\xi$ we deduce from \eqref{2.15}--\eqref{2.16} that if we
set 
$$\Phi(t,x,y,\xi)=\varphi(x,y,t,\xi)+t$$
then on $\text{supp }a$ we must have that the mixed Hessian of $\Phi$ with respect to the $n+1$ 
variables $(t,\xi)$ satisfies
$$\det \,  \Bigl(\frac{\partial^2 \Phi}{\partial(t,\xi)\partial(t,\xi)}\Bigr)\ne 0 \quad \text{if } \, \, \nabla_\xi\Phi =0.$$
Consequently, if $a$ is as in \eqref{2.12} and if $b(t) \in C^\infty_0(\R)$ vanishes outside of a small neighborhood of  $t_0$, we conclude from
the method of stationary phase that
\begin{align*}
\iint \bigl(1-\beta(t)\bigr) &\Hat \rho(t/T) e^{it\la}e^{i\varphi(x,y,t,\xi)} b(t)a(t,x,y,\xi) \, d\xi dt
\\
&=\la^n \int_{{\mathbb R}^{n+1}} \bigl(1-\beta(t)\bigr) \Hat \rho(t/T) b(t)a(t,x,y,\la\xi) \, e^{i\la \Phi(x,y,t,\xi)} \, d\xi dt
\\
&=O(\la^n\la^{-\frac{n+1}2})=O(\la^{\frac{n-1}2}).
\end{align*}

Since $2\pi T$ times the term $II$ in \eqref{2.7} can be written as the sum of finitely many terms
of this form (depending on $T$) (and an $O(\la^{-N})$ term coming from the smooth errors in the
parametrix), we deduce that \eqref{2.9} must be valid, which completes the proof of Lemma~\ref{lemma2.2}.
\end{proof}

We can now turn to the proof of Theorem~\ref{theorem2.1}.  We shall adapt an argument from \cite{Scrit}
which uses an idea of Bourgain~\cite{BKak} and Lorentz space estimates of Bak and Seeger~\cite{BakSeeg}
for the operators in \eqref{2.5} corresponding to $T=1$.

\begin{proof}[Proof of Theorem~\ref{theorem2.1}]
Since $\mu(p_c)=1/p_c$, proving \eqref{2.2} is equivalent
to showing that if $E_\la: L^2(M)\to L^2(M)$ is the
projection onto the eigenspace with eigenvalue $\la$ then
\begin{equation}\label{2.2'}
\|E_\la\|_{L^2(M)\to L^{p_c}(M)}=o(\la^{\frac1{p_c}}).
\tag{2.2$'$}
\end{equation}

To do this we shall use an estimate of Bak and
Seeger~\cite{BakSeeg} that will allow us to deduce
\eqref{2.2'} from the easier weak-type estimates
\begin{equation}\label{2.2''}\tag{2.2$''$}
\|E_\la\|_{L^2(M)\to L^{p_c,\infty}(M)}
=o(\la^{\frac1{p_c}}).
\end{equation}
Indeed Bak and Seeger showed that if $\chi_\la$ denote
the standard spectral projection operators
$$\chi_\la f=\sum_{\la_j\in [\la,\la+1)}E_{\la_j}f,$$
then on any $(M,g)$ 
one has the Lorentz space estimates
$$\|\chi_\la\|_{L^2(M)\to L^{p_c,2}(M)}=O(\la^{\frac1{p_c}}).$$
Since $E_\la$ is a projection operator and $E_\la
=\chi_\la\circ E_\la$ this implies that
\begin{equation}\label{2.17}
\|E_\la\|_{L^2(M)\to L^{p_c,2}(M)}=O(\la^{\frac1{p_c}}).
\end{equation}
If one interpolates between \eqref{2.17}
and \eqref{2.2''} one obtains \eqref{2.2'}
(see e.g. Chapter V in Stein and Weiss \cite{SteinWeiss}
or \S 4 in \cite{Scrit}).

Let us rewrite \eqref{2.2''}.  Given $f$ with $L^2(M)$ norm one, we shall let
$$\omega_{E_\la f}=\bigl|\bigr\{x\in M: \, 
|E_\la f(x)|>\alpha \bigr\}\bigr|, \quad \alpha>0,$$
denote the distribution function of $E_\la f$, with
$|U|$ denoting the $dV_g$ measure of $U\subset M$.
Then \eqref{2.2''} is just the statement that for
any fixed $\e>0$ we can find a $\Lambda_\e<\infty$ so that
\begin{equation}\label{2.18}
\omega_{E_\la f}(\alpha)\le \e \la
\alpha^{-\frac{2(n+1)}{n-1}}, \quad
\text{if } \, \, \|f\|_{L^2(M)}=1 \, \, 
\, \text{and } \, \, \la \ge \Lambda_\e.
\end{equation}

To prove this, we first note that because we are
assuming \eqref{1.9}, by Theorem~\ref{thm1.2}
since $2<\tfrac{2n}{n-1}<p_c$ and since
$\mu(\tfrac{2n}{n-1})=\tfrac12 \cdot \tfrac{n-1}{2n}$,
we have
$$\|E_\la \|_{L^2(M)\to L^{\frac{2n}{n-1}}(M)}
=o(\la^{\frac12\cdot \frac{n-1}{2n}}).$$
Therefore, by Chebyshev's inequality, given
$\delta>0$ we can find a $\Lambda_\delta<\infty$ so that
\begin{equation}\label{2.19}
\omega_{E_\la f}(\alpha)\le \delta \la^{\frac12}
\alpha^{-\frac{2n}{n-1}}, \quad
\text{if } \, \, \|f\|_{L^2(M)}=1 \, \, 
\, \text{and } \, \, \la \ge \Lambda_\delta.
\end{equation}

A calculation shows that these bounds yield
\eqref{2.18} if $\alpha$ satisfies
\begin{equation}\label{2.19a}
\alpha \le \la^{\frac{n-1}4}\bigl(\e/\delta\bigr)^{n-1}.
\end{equation}
We shall specify $\delta=\delta(\e)$ in a moment.

We are also assuming that \eqref{1.6} is valid
and hence, by Theorem~\ref{thm1.1},
$$\|E_\la\|_{L^2(M)\to L^\infty(M)}=o(\la^{\frac{n-1}2}).$$
 Thus, given $\delta>0$ as above we have that
there must be a $\Lambda_\delta<\infty$ so that
$$\|E_\la f\|_{L^\infty(M)}\le \delta\la^{\frac{n-1}2},
\quad \text{if } \, \, \|f\|_{L^2(M)}=1
\, \, \, \text{and } \, \, \la \ge \Lambda_\delta.$$
This means that for such $\la$ we have
\begin{equation}\label{2.20}
\omega_{E_\la f}(\alpha)=0 \quad \text{if }
\, \, \alpha \ge \delta \la^{\frac{n-1}2}.
\end{equation}

By \eqref{2.19a} and \eqref{2.20}, we have reduced
matters to showing that for large $\la$ we have
\begin{multline}\label{2.18'}\tag{2.18$'$}
|\{|E_\la f(x)|>\alpha\}|
=\omega_{E_\la f}(\alpha)\le \e \la \alpha^{-\frac{2(n+1)}{n-1}},
\\
 \text{if } \, \, \|f\|_{L^2(M)}=1 \, \, 
\text{and } \, \, 
\alpha \in I_{\e,\delta}=\bigl(
\la^{\frac{n-1}4}(\e/\delta)^{n-1},\delta \la^{\frac{n-1}2}\bigr).
\end{multline}

To prove this, we note that if $\rho$ is as in \eqref{2.3} then for any $T\ge 1$ we have
$$E_\la f=\rho(T(\la-P))E_\la f
\quad \text{and } \, \, \|E_\la f\|_{L^2(M)}\le
\|f\|_{L^2(M)}.
$$
As a result, we would have \eqref{2.18'} if we
could choose $T=T_\e \gg 1$ so that for large 
$\la$ we have
\begin{multline}\label{2.18''}\tag{2.18$''$}
\bigl|\bigl\{x: \, 
|\rho(T_\e(\la-P))h(x)|>\alpha\bigr\}\bigr|
\le \e \la \alpha^{-\frac{2(n+1)}{n-1}},
\\\ \text{if } \, \, \|h\|_{L^2(M)}=1
\, \, \text{and } \, \, \alpha \in I_{\e,\delta}.
\end{multline}

To prove this, as in \cite{Scrit} we shall adapt
an argument of Bourgain~\cite{BKak} which exploits
the bounds in Lemma~\ref{lemma2.2} that are based
on our other assumption that $U(t)$, $t\ne 0$, has
no caustics.  To do so put
\begin{equation}\label{2.21}
r=\la\alpha^{-\frac4{n-1}}T_\e^{\frac{2}{n-1}}.
\end{equation}
Note then that
\begin{equation}\label{2.22}
r\ge \la^{-1} \quad\text{if }\, \,
\delta \le T^{-\frac12}_\e \, \, \text{and }
\, \, \alpha\in I_{\e,\delta},
\end{equation}
due to the fact that $\alpha \le \la^{\frac{n-1}2}\delta$
if $\alpha\in I_{\e,\delta}$.  Since the $\delta$
in \eqref{2.19} and \eqref{2.21} can be made
arbitrarily small, we shall assume that we have this condition after we specify $T_\e$ in a bit.

Let $A=A_\alpha =|\{\, |\rho(T_\e(\la-P))h(x)|>\alpha
\, \}|$ denote the set in \eqref{2.18''}.
Then we are trying to show that $|A|$ satisfies
the bounds there assuming that $\alpha\in I_{\e,\delta}$.
At the expense of replacing $A$ by a set of proportional
measure, we may assume that
\begin{equation}\label{2.23}
A=\bigcup_j A_j \quad
\text{where } \, \, \text{diam }A_j\le r
\, \, \text{and }\, \, d_g(A_j,A_k)\ge C_0 r, \, 
\, j\ne k,
\end{equation}
where $r$ is as in \eqref{2.21}--\eqref{2.22} and
$C_0$ will be specified shortly.  Here $\text{diam }U$
denotes the diameter of $U\subset M$ as measured
by the Riemannian distance function.

In addition to \eqref{2.6} we shall also require
the following simple estimate from \cite{SHR} and
\cite{Scrit}, which says that for $T,\la \ge 1$
we have the uniform bounds
\begin{equation}\label{2.24}
\|\rho(T(\la-P))f\|_{L^2(B_r(x))}\le
Cr^{\frac12}\|f\|_{L^2(M)}
\quad \text{if }\, \, \la^{-1}\le r\le \text{Inj }M.
\end{equation}

If 
$$\psi_\la(x)
=\begin{cases}
\rho(T_\e(\la-P))h(x)/|\rho(T_\e(\la-P))h(x)|, \, \, 
\text{if } \, \, \rho(T_\e(\la-P))h(x)\ne 0
\\ 
1, \, \, \, \text{otherwise }
\end{cases}
$$
denotes the signum function of $\rho(T_\e(\la-P))h$
and if we let $a_j$ be $\psi_\la$ times the indicator
function, $\1_{A_j}$, of the set $A_j$ as above,
then since we are assuming that $\|h\|_2= 1$, by
Chebyshev's inequality and the Cauchy-Schwarz inequality
\begin{align*}
\alpha|A|&\le \Bigl| \, \int \sum_j\rho(T_\e(\la-P))h
\, \overline{\psi_\la \1_{A_j}} \, dV_g \, \Bigr|
\\
&\le \Bigl(\, \int \bigl| \sum_j 
\rho(T_\e(\la-P))a_j\bigr|^2 \, dV_g \, \Bigr)^{1/2}.
\end{align*}
As a result, if 
$$S_\la =\rho(T_\e(\la-P))^*
\circ \rho(T_\e(\la-P))=|\rho|^2(T_\e(\la-P)),$$
then
\begin{align}\label{2.25}
\alpha^2|A|^2 &\le \sum_j \int\bigl|
\rho(T_\e(\la-P))a_j\bigr|^2 \, dV_g
+\sum_{j\ne k} \int\rho(T_\e(\la-P))a_j
\, \overline{\rho(T_\e(\la-P))a_k}\, dV_g
\\
&=\sum_j \int|\rho((T_\e(\la-P))a_j|^2 \, dV_g
+\sum_{j\ne k} \int S_\la a_j \, \overline{a_k}\, dV_g
\notag
\\
&=I+II. \notag
\end{align}

By \eqref{2.21}--\eqref{2.23} and the dual version of
\eqref{2.24}
$$I\le Cr\sum_j \int |a_j|^2 \, dV_g
=Cr|A|=C\la\alpha^{-\frac4{n-1}}T_\e^{-\frac2{n-1}}|A|.
$$
Thus if we let
$$T_\e =\bigl(\e/2C\bigr)^{\frac{n-1}2},
$$
we have 
\begin{equation}\label{2.26}
I\le \frac12 \e \la \alpha^{-\frac4{n-1}}|A|.
\end{equation}

To estimate the other term, $II$, in \eqref{2.25},
we note that since we have finally specified $T_\e$,
by Lemma~\ref{lemma2.2} we have that the kernel
$S_\la(x,y)$ of $S_\la$ satisfies
$$|S_\la(x,y)|\le CT_\e^{-1}\bigl(\la/d_g(x,y)\bigr)^{\frac{n-1}2}+C_\e \la^{\frac{n-1}2},$$
due to the fact that the Schwartz class function
$|\rho|^2$ equals one at the origin and has compactly supported
Fourier transform by \eqref{2.3}.

Therefore, by \eqref{2.21} and \eqref{2.23},
\begin{align}\label{2.27}
II&\le \bigl[ CT^{-1}_\e (\la/C_0 r)^{\frac{n-1}2}
+C_\e \la^{\frac{n-1}2}\bigr] \sum_{j\ne k}\|a_j\|_{L^1}
\|a_k\|_{L^1}
\\
&\le CC_0^{-\frac{n-1}2}\alpha^2|A|^2+C_\e 
\la^{\frac{n-1}2}|A|^2.
\notag
\end{align}
Since we are assuming in \eqref{2.18''} that
$\alpha\in I_{\e,\delta}$ and hence
$\alpha\ge \la^{\frac{n-1}4}(\e/\delta)^{n-1}$, 
we can control the last term as follows
$$C_\e\la^{\frac{n-1}2}|A|^2
\le C_\e (\delta/\e)^{2(n-1)}\alpha^2|A|^2
\le \frac14\alpha^2|A|^2
\quad \text{if } \, \,
C_\e(\delta/\e)^{2(n-1)}\le \frac14.$$
Since the $\delta$ in \eqref{2.19a} and \eqref{2.20}
can be taken to be as small as we like, because
we are assuming \eqref{1.6} and \eqref{1.9},
we can fix such a $\delta$ which also satisfies
the condition in \eqref{2.22} and obtain this
bound for the last term in \eqref{2.27}.
As a result, if we choose the constant $C_0$ in \eqref{2.23} so that $CC_0^{-\frac{n-1}2}=\tfrac14$,
then by \eqref{2.27} we have
$$II\le \frac12 \alpha^2|A|^2.$$

If we combine this with \eqref{2.25} and \eqref{2.26}
we deduce that for large enough $\la$ we have
$$\alpha^2|A|^2\le \e \la \alpha^{-\frac4{n-1}}|A|.$$
Since this is equivalent to the statement that
$$|A|\le \e\la \alpha^{-\frac{2(n+1)}{n-1}},$$
the proof of \eqref{2.18''} and hence that of
Theorem~\ref{theorem2.1} is complete.
\end{proof}

%
%
%
%
%
%
%
%
%
%
%
%
%
%
%
%
%


\bibliography{EF}{}
\bibliographystyle{amsplain}
\end{document}